\documentclass[a4paper]{amsart}

\usepackage{amssymb}
\usepackage{amsmath}
\usepackage{mathrsfs}
\usepackage{latexsym}
\usepackage{stmaryrd}
\usepackage{tikz}
\usetikzlibrary{matrix,arrows,decorations.pathmorphing, calc, positioning}
\usepackage{tikz-cd}
\usepackage{mathtools}
\usepackage{amsthm}
\usepackage{multicol}
\usepackage{relsize}
\usepackage{exscale}
\usepackage{nomencl}
\usepackage{enumerate}
\usepackage{hyperref}
\usepackage[capitalize]{cleveref}

\newcommand{\st}{\,|\,}
\newcommand{\isom}{\cong}
\newcommand{\eps}{\varepsilon}

\newcommand{\C}{\mathbb{C}}
\newcommand{\Z}{\mathbb{Z}}

\newcommand{\uodot}{\boxdot}
\newcommand{\cuodot}{\widehat{\uodot}}
\newcommand{\ucoprod}{\boxtimes}
\newcommand{\cucoprod}{\widehat{\ucoprod}}

\DeclareMathOperator{\Indecomp}{Ind}
\DeclareMathOperator{\Prim}{Add}

\newcommand{\Ab}{\mathbf{Ab}}

\newcommand{\Ho}{\mathbf{Ho}}
\newcommand{\Set}{\mathbf{Set}}

\newcommand{\alg}{\mathbf{Alg}}
\newcommand{\Alg}{\mathbf{Alg}}

\newcommand{\CAlg}{\mathbf{CAlg}}

\newcommand{\CBiring}{\mathbf{CBiring}}
\newcommand{\CPleth}{\mathbf{CPlethory}}

\newcommand{\LoopPleth}{\mathbf{\Omega Plethory}}

\newcommand{\cotimes}{\widehat{\otimes}}
\newcommand{\codot}{\widehat{\odot}}
\newcommand{\cE}{\widehat{E}}

\newcommand{\rE}{\underline{E}} % representing space
\newcommand{\rK}{\underline{K}} % representing space
\DeclareMathOperator{\Spec}{Spec}
\DeclareMathOperator{\Spf}{Spf}

\theoremstyle{plain}
\newtheorem{thm}{Theorem}[section]
\newtheorem{prop}[thm]{Proposition}

\newtheorem{cor}[thm]{Corollary}

\theoremstyle{definition}
\newtheorem{df}[thm]{Definition}
\newtheorem{ex}[thm]{Example}

\usepackage[textwidth=14cm,hcentering]{geometry}

\title{The plethory of operations in complex topological $K$-theory}
\author{William Mycroft}
\address[W.~Mycroft]{Flat 202,
Baldwin Point\\
6 Sayer Street\\
London\\
SE17 1FG\\ England}
\email{william.mycroft@gmail.com}

\author{Sarah Whitehouse}
\address[S.~Whitehouse]{
School of Mathematics and Statistics\\
University of Sheffield\\ S3 7RH\\ England}
\email{s.whitehouse@sheffield.ac.uk }
\date{\today}

 % K-theory operations and generalized cohomology operations


\begin{document}

\maketitle

\begin{abstract}
We provide a concrete introduction to the topologised, graded analogue of an algebraic structure known as a plethory, originally due to Tall and Wraith. Stacey and Whitehouse showed this structure is present on the cohomology operations for a suitable generalised cohomology theory.
We compute an explicit expression for the plethory of operations for complex topological $K$-theory.
This is formulated in terms of a plethory enhanced with structure corresponding to the looping of operations.
In this context we show that the familiar $\lambda$-operations generate all the operations.
\end{abstract}

\setcounter{tocdepth}{1}
%\tableofcontents

\section{Introduction}

Cohomology operations provide a very powerful piece of structure associated with a generalised cohomology theory and over the years they been used to  prove many deep results in algebraic topology. However, despite the ubiquity of cohomology operations, there are some challenges in identifying the appropriate algebraic framework in which to encode the rich structure the operations admit.

Historically, \emph{cooperations}, the homological analogue of operations, have often been the objects of interest and in many cases of interest they encode the same information. These are amenable to study via the medium of Hopf rings and many useful results have been proved that way. One may also consider operations from one theory to another, again with corresponding Hopf rings. For example, under mild hypotheses, Hunton gives a general description of the homology Hopf ring associated to the completion of a theory with respect to an ideal in the coefficient ring~\cite{Hunton:2002}.

Unfortunately, a Hopf ring contains no algebraic structure which naturally corresponds to composition of operations. To address this, Boardman, Johnson and Wilson introduced the notion of an \emph{enriched Hopf ring} which enhances the structure with an external action encoding the missing information \cite{BoardmanJ:1995}. Enriched Hopf rings of cooperations have been computed for many interesting theories,
including complex $K$-theory~\cite[Theorem 17.14]{BoardmanJ:1995}. However,
 this structure can be somewhat cumbersome for performing computations as the enrichment is not easily expressed in terms of generators and relations.

An alternative approach proves fruitful. Roughly speaking, operations act non-linearly on cohomology algebras, and this determines the structure in the same way that (not necessarily commutative) $k$-algebras are precisely the objects which act on $k$-modules. The relevant abstract algebraic structure was first introduced in 1970 by Tall and Wraith~\cite{TallW:1970}, and subsequently
 studied by Bergman and Hausknecht~\cite{BergmanH:1996}, and
  by Borger and Wieland~\cite{BorgerW:2005} who coined the term \emph{plethory}. A priori, cohomology operations do not naturally fit into this framework due to considerations of grading and topologies on the algebraic structures. However, as shown by Stacey and Whitehouse~\cite{StaceyW:2009}, in sufficiently nice cases the operations admit the structure of a graded topologised plethory and this acts on completed cohomology algebras. 
Here the completion is with respect to the skeletal filtration, as discussed in~\cite[Section 3]{Boardman:1995}.
A related approach by Bauer considers
  \emph{formal plethories}, thus avoiding completion issues~\cite{Bauer:2014}. All this should be viewed as an algebraic shadow of corresponding structure in the world of spectra and there is current work towards developing a theory of spectral plethories.

Of course, in the case of complex topological $K$-theory, there is a long tradition of work with cohomology operations,
often formulated in terms of $\lambda$-operations or Adams operations. The ring of symmetric functions provides
a basic example of a plethory, whose algebras are $\lambda$-rings.
Yau has related the enriched Hopf ring approach to that of filtered $\lambda$-rings~\cite{Yau:2003}, restricting attention to the degree zero part of
complex $K$-theory. Working with $p$-adic coefficients, Bousfield's theory of $p$-adic $\theta$-rings captures the structure~\cite{Bousfield:1996}, and work of Rezk, again in a $p$-complete setting, extends this to exhibit the relevance of plethories to power operations at higher chromatic heights~\cite{Rezk:2009}.

The main aim of this paper is to  give a concise full description of the \emph{integral} operations  of complex topological $K$-theory in plethystic terms. We first give a direct proof of the application of plethories to cohomology operations which illuminates exactly where topological issues arise. We then extend our algebraic gadgets to encode the \emph{looping} of operations arising in the topological context. Applying our technical framework to the study of the operations of complex topological $K$-theory yields our main result, in particular showing how the $\lambda$-operations generate all $K$-theory operations.

The main result is  Theorem~\ref{thm:main}. This describes the operations as a  $\Z/2$-graded $\Z$-plethory with looping,
in terms of the plethory of symmetric functions and the plethory of set maps from $\Z$ to $\Z$.
\medskip

This paper is organised as follows. Section~\ref{sec:toppleth} covers plethories in a graded and topologised context.
The (completed) plethory structure of set maps from a ring to itself, such as
$\Set(\Z,\Z)$, is discussed here.
 Section~\ref{sec:ungddops} covers
the plethory of operations in ungraded $K$-theory and looping is discussed in Section~\ref{sec:looping}.
The main result
appears in Section~\ref{sec:mainresult}.
\medskip

Throughout, rings and algebras will be assumed to be (graded) commutative and unital unless stated otherwise.

\section{Topological plethories}
\label{sec:toppleth}

We generalise the theory of plethories \cite{TallW:1970, BorgerW:2005} to a suitably graded and topologised context. This variant is needed to  capture the structure on cohomology operations. We assume familiarity
with~\cite{BorgerW:2005} and our focus is on the differences  in the graded, topologised case.

\medskip

Fix a commutative monoid $Z$, typically $(\Z, +)$ or $(\Z/2, +)$, used for grading.
\medskip

Let $k$ and $k'$ be $Z$-graded rings. Let $\alg_k$ be the category of $k$-algebras and
let $\CAlg_k$ be the category of filtered $k$-algebras which are complete Hausdorff under the filtration topology, meaning that the completion map $A\to \widehat{A}=\displaystyle \lim_{\leftarrow} A/F^aA$ is an isomorphism.
Morphisms are continuous $k$-algebra maps of degree zero. We write $\cotimes$ for the completed tensor product over $k$.
Further details can be found in~\cite[Section 6]{Boardman:1995}.

\begin{df}
The category $\CBiring_{k,k'}$ of \emph{complete Hausdorff $k$-$k'$-birings} is the category of co-$k'$-algebra objects in $\CAlg_k$.
 To be explicit, an object in this category consists of a $Z$-graded collection of complete Hausdorff $k$-algebras $B_\bullet = (B_n)_{n \in Z}$ together with continuous $k$-algebra maps for each $n \in Z$
	\begin{align}
	\Delta^{+} & \colon B_n \to B_n \cotimes B_n \tag{co-addition} \\
	\eps^{+} & \colon B_n \to k \tag{co-zero} \\
	\sigma & \colon B_n \to B_n \tag{co-additive inverse} \\
	\Delta^{\times} & \colon B_n \to \prod_{i+j=n} B_i \cotimes B_j \tag{co-multiplication} \\
	\eps^{\times} & \colon B_0 \to k \tag{co-unit}  \\
	\intertext{and for each $\kappa \in k'$,}
	\gamma(\kappa) & \colon B \to B \tag{co-$k'$-linear structure}
	\end{align}
	satisfying the usual relations for a co-$k'$-algebra object \cite{TallW:1970, Boardman:1995}.
\end{df}
\medskip	
	We use the above notation for the co-algebraic structure maps of a biring throughout this paper.
	
	It is customary~\cite{Boardman:1995, BorgerW:2005} to consider the equivalent formulation of the co-$k'$-linear structure given by $\beta(\kappa) =  \eps^\times \circ \gamma(\kappa)$, where we set $\eps^{\times}(b)=0$ for $b\in B_n$ with $n\neq 0$. Endowing $\alg_k(B_\bullet, k)$ with the $Z$-graded ring structure determined by the other maps above, this yields a map of $Z$-graded rings $\beta \colon k' \to \alg_k(B_\bullet, k)$. (However, this alternative description is unavailable in the case of
	non-(co-unital) co-$k'$-algebra objects, where we have no $\eps^{\times}$.)
	
	A complete Hausdorff $k$-$k'$-biring $B$ is naturally $Z$-$Z$-bigraded, with gradings induced by the gradings on $k$ and $k'$. By an element $x \in B$, we mean $x \in B_n$ for some $n \in Z$. For $x\in B_n$, we define the \emph{$\bullet$-degree} by $\deg_\bullet(x) = n$ and the \emph{$*$-degree} to be $\deg_*(x) = |x| \in Z$, the degree of $x$ in the graded $k$-algebra $B_n$. We can recover the ungraded context as the special case $Z = 0$, the trivial monoid.
	
	We make extensive use of \emph{sumless Sweedler notation} \cite{Sweedler:1969}, writing
\begin{align*}
\Delta^+(x) &= x_{(1)} \otimes x_{(2)},\\
\Delta^\times(x) &= x_{[1]} \otimes x_{[2]}.
\end{align*}

\medskip

In the untopologised case, the algebro-geometric viewpoint of $k$-$k'$-birings as representable functors $\Alg_k \to \Alg_{k'}$ turns out to give very useful intuition. This naturally generalises to the topologised setting via the language of formal schemes. We only need affine schemes, so we use the following definitions~\cite{Strickland:2000}, where we use varieties of algebras in the sense of universal algebra~\cite{Bergman:2015}.

\begin{df}
A \emph{$k$-scheme} is a covariant representable functor $X \colon \Alg_k \to \Set$. Given a variety of algebras $\mathcal{V}$, if $X$ has a specified lift to a functor
$\Alg_k \to \mathcal{V}$,  we say the lift, often also denoted $X$,
 is a \emph{$\mathcal{V}$-algebra $k$-scheme}. If $A$ denotes the representing $k$-algebra, we write $X = \Spec_k(A) = \Alg_k(A, -)$.

A \emph{formal $k$-scheme} is a functor $X \colon \Alg_k \to \Set$ which is a small filtered colimit of $k$-schemes. If $X$ has a specified lift to a functor $\Alg_k \to \mathcal{V}$, we say the lift, often also denoted $X$, is a
\emph{$\mathcal{V}$-algebra formal $k$-scheme}. Given a  filtered  $k$-algebra $A$, we define the formal $k$-scheme
$$ \Spf_k(A) = \varinjlim_a \Alg_k(A/F^aA, -). $$
This construction is functorial, giving a contravariant functor $\Spf_k$ from filtered $k$-algebras to formal $k$-schemes.
\end{df}

It is worth noting that $\Spf_k(A) = \Spf_k(\widehat{A})$, i.e. $\Spf_k(-)$ is blind to completions.
Without going into detail, we remark that both $k$-schemes and formal $k$-schemes preserve completeness and Hausdorff properties; see~\cite{StaceyW:2009}. For example, if $X \colon \Alg_k \to \Alg_{k'}$ is a (formal) $k'$-algebra $k$-scheme then $X$ restricts to a functor $\CAlg_k \to \CAlg_{k'}$.

\begin{df}
A formal $k$-scheme $X$ is \emph{solid} if it is isomorphic to $\Spf_k(A)$ for some filtered $k$-algebra $A$.
\end{df}

\begin{prop}
The functor $\Spf_k$ gives an anti-equivalence between complete Hausdorff $k$-$k'$-birings $B$ and solid formal $k'$-algebra $k$-schemes.
\end{prop}
\begin{proof} The functor $\Spf_k$ is an anti-equivalence from $\CAlg_k$ to the category of solid formal $k$-schemes and the result follows by restricting to co-$k'$-algebra objects.\end{proof}

\begin{ex}\label{ex:zero}
As in~\cite[Example 1.2(1)]{BorgerW:2005}, $k$ is the initial $k$-$k'$-biring, with all structure maps given by the identity map of $k$. The corresponding functor is the constant
functor at the zero ring.
\end{ex}

\begin{ex}\label{ex:id}
	When $Z = \Z$, the identity functor $\Alg_k \to \Alg_k$ can be expressed as $\Spf_k(\mathcal{I})$ where
	$$ \mathcal{I}_n = \begin{cases}
		k[\iota_n] & n \text{ even} \\
		\Lambda_k[\iota_n] & n \text{ odd},
	\end{cases} $$
	$|\iota_n| = n$  and each $\mathcal{I}_n$ has the discrete topology. The structure maps are given by
	\begin{align*}
	\Delta^+(\iota_n) &= 1 \otimes \iota_n + \iota_n \otimes 1 \\
	\eps^{+}(\iota_n) &= 0 \\
	\sigma(\iota_n) &= -\iota_n \\
	\Delta^\times(\iota_n) &= \sum_{r+s=n} \iota_r \otimes \iota_s \\
	\eps^{\times}(\iota_n) &= \begin{cases}
	1 & n=0 \\
	0 & \text{otherwise}
	\end{cases}  \\
	\intertext{and for each $\kappa \in k$,}
	\beta(\kappa)(\iota_n) &= \begin{cases}
	\kappa & |\kappa|=n \\
	0 & \text{otherwise.}
	\end{cases}
	\end{align*}
\end{ex}

A special case of the following example is relevant to our main result.

\begin{ex}
	The collection of set maps  $\Set(k',k)$ endowed with the topology arising from the pro-finite filtration
\[\left\{\ker\left(\Set(k', k) \to \Set(k'_a, k) \right) \,|\, k'_a \subseteq k', k'_a \text{ finite subring} \right\}\]
 naturally admits the structure of an ungraded complete Hausdorff $k$-$k'$-biring. The $k$-algebra structure is induced by the $k$-algebra structure on $k$ and the co-$k'$-algebra structure is induced by the $k'$-algebra structure on $k'$.

For example, the co-addition is given by the map
	$$ \Set(k', k) \xrightarrow{\Set(+, k)} \Set(k' \times k', k) \isom \Set(k', k) \cotimes \Set(k', k). $$
	
	The formal $k'$-algebra $k$-scheme $\Spf_k(\Set(k', k))$ is naturally isomorphic to the functor of \emph{complete orthogonal idempotents} given on $k$-algebras by
	$$ COI_{k'}(A) = \left\{ (x_i) \in \prod_{i \in k'} A \,\left|\right.\, \sum_i x_i = 1, x_i^2 = x_i, x_ix_j = 0 \text{ for } i \neq j \right\}. $$
	The addition and multiplication are specified by
	\begin{align*}
		\pi_l((x_i) + (y_j)) &= \sum_{i+j=l} x_i y_j \\
		\pi_l((x_i)(y_j)) &= \sum_{ij=l} x_i y_j,
	\end{align*}
	where $l\in k'$ and
	$\pi_l$ denotes the canonical projection $\prod_{i \in k'} A \to A$ to the component indexed by $l$.
	The zero in $COI_{k'}(A)$ is $(\delta_{i0})_{i \in k'}$ and the $1$ is $(\delta_{i1})_{i \in k'}$, where $\delta_{id}$ is the Kronecker delta function.	
	The topology is given by the filtration ideals consisting of sequences containing finitely many non-zero elements. The identification 
\[ \Spf_k(\Set(k', k)) \cong COI_{k'}(-)
\] is given by the natural isomorphism which sends $\chi_d$ to $(\delta_{id})_{i \in k'}$, where  $\chi_d$ is the indicator functor on $\{d\} \subseteq k'$.
	
	When $A$ contains no zero divisors, we have $COI_{k'}(A) \isom k'$. In fact, $\Spf_k(\Set(k', k))$ is the nearest solid formal $k'$-scheme to the constant $k'$-algebra scheme $A \mapsto k'$; see~\cite[Section~4]{Bauer:2014}.
\end{ex}

In various applications, we frequently encounter \emph{non-(co-unital) $k$-$k'$-birings}, corresponding to representable functors from $\Alg_k$ to $\Alg_{k'}^!$, the category of non-unital $k'$-algebras or their topological generalisations. At the level of algebras, it is standard to remedy the lack of a unit via \emph{unitalisation}: given a non-unital $k$-algebra $R$, one forms the $k$-module $k \oplus R$ together with the obvious multiplication. More generally, if $S$ is unital and $R$ is additionally an $S$-module then the coproduct of $k$-modules $S \oplus R$ is naturally a unital $k$-algebra with multiplication given by
$$(s_1 + r_1)(s_2 + r_2) = (s_1 s_2 + r_1 r_2 + s_1 \cdot r_2 + s_2 \cdot r_1)$$
and unit $1_S + 0_R$.
	
	This construction has an analogue in the context of algebra schemes. If $\Spec_k(B)$ is a representable non-unital $k'$-algebra scheme, $\Spec_k(B')$ is a unital $k'$-algebra scheme, and $\Spec_k(B)(A)$ is naturally a $\Spec_k(B')(A)$-module,
	 then the functor $\Spec_k(B' \otimes B)$ is a unital $k'$-algebra scheme given, up to natural isomorphism, on objects by
	$$ A \mapsto \Spec_k(B')(A) \oplus \Spec_k(B)(A). $$
	
	At the level of the representing objects this translates to a $B'$-comodule structure on $B$.
The comultiplication $\Delta^\times$ on $B' \otimes B$ is given by the image of the identity map of
$B' \otimes B\otimes B' \otimes B$ under the composite
\[
\Spec_k(B' \otimes B\otimes B' \otimes B)(A)\cong \Spec_k(B' \otimes B)(A)\times \Spec_k(B' \otimes B)(A)
\xrightarrow{\mu}  \Spec_k(B' \otimes B)(A),
\]
where $\mu$ is the multiplication and $A=B' \otimes B\otimes B' \otimes B$. Using this, one can compute an explicit formula for $\Delta^\times$ and similarly for the counit $\epsilon^\times$.
	Denoting the coaction $B \to B' \otimes B$ by $y \mapsto y_{\{1\}} \otimes y_{\{2\}}$, we
	find the following formulas for the comultiplication and counit of
	 the $k$-$k'$-biring structure on $B' \otimes B$.
	\begin{align*}
	\Delta^\times(x \otimes y) &= x_{[1]} y_{(2) \{1\}} \otimes y_{(1)[1]} y_{(3) \{2\}} \otimes x_{[2]}y_{(3)\{1\}} \otimes y_{(1)[2]} y_{(2)\{2\}} \\
	\eps^\times(x \otimes y) &= \eps^\times(x) \eps^+(y).
	\end{align*}

This construction generalises without difficulty to our topologised framework, replacing schemes with formal schemes and completing tensor products.

\begin{ex}
	\label{ex:biring-adjoin-unit}
	Let $B$ be a non-(co-unital) $k$-$k$-biring and let $A$ be a $k$-algebra. The non-unital $k$-algebra
	$\Spf_k(B)(A)$ naturally admits a $\Spf_k(\Set(k,k))(A)$-module structure which, after identifying $\Spf_k(\Set(k,k))(A)$ with $COI_k(A)$, is given by 		
	$$ ((a_\lambda) \cdot \phi)(b) = \sum_\lambda \phi(\gamma(\lambda)(b)) a_\lambda. $$
Here $\phi\in	\Spf_k(B)(A)$, $\lambda\in k$, $(a_\lambda)\in COI_k(A)$, $b\in B$ and $\gamma$ specifies the co-$k$-linear structure of $B$.

	This translates to a $\Set(k,k)$-comodule structure on $B$ given by
	$$ b \mapsto  \sum_\lambda \chi_\lambda \otimes \gamma (\lambda)(b) $$
	and thus, $\Set(k,k) \cotimes B$ is naturally a $k$-$k$-biring with structure maps $\Delta^\times$, $\eps^\times$ determined by
	\begin{align*}
	\Delta^\times(\chi_d \otimes b) &= \sum_{rs=d} \chi_r \otimes b_{(1)[1]} \gamma(s)(b_{(2)}) \otimes \chi_s \otimes b_{(1)[2]}  \gamma(r)(b_{(3)}) \\
	\eps^\times(f \otimes b) &= \eps^\times(f) \eps^+(b).
	\end{align*}
\end{ex}
\medskip

We generalise the composition product $\odot$~\cite{BorgerW:2005, TallW:1970}, which represents the composition of functors, to the graded topologised setting in two stages, first adding the grading and then the topology. The grading will allow us to model operations between graded objects and the topological setting allows us to consider only the continuous operations.

Just as with the tensor product of algebras, the composition product of a complete Hausdorff biring with a complete Hausdorff algebra is not necessarily complete Hausdorff with respect to the canonical filtration. As with the tensor product, this is remedied by taking the completion.

\begin{df}
	For a complete Hausdorff $k$-$k'$-biring $B$ and complete Hausdorff $k'$-algebra $A$, we define the \emph{complete Hausdorff composition product} $\codot$
 as follows. First, take the quotient of
	$ B \odot A$ by the ideal generated by the relations $b \odot a = 0$ whenever $\deg_\bullet(b) \neq |a|$. The grading on $B \odot A$ is specified by $|b \odot a| = \deg_*(b) = |b|$. Now define
 $B \codot A$ to be the complete Hausdorff $k$-algebra
	$$ B \codot A = \varprojlim_{\alpha, \beta} \frac{B}{F^\beta B} \odot \frac{A}{F^\alpha A}$$
	together with the canonical filtration where $F^\alpha A$ and $F^\beta B$ denote the filtration ideals on $A$ and $B$ respectively.
\end{df}

The defining properties of the composition product generalise without difficulty to the graded topological setting and we have a bifunctor
 $\codot:\CBiring_{k,k'} \times  \CAlg_{k'} \to \CAlg_k$.
 If $B$ and $B'$ are complete Hausdorff birings, the bigrading on the composition product is defined $\bullet$-componentwise in the sense that $(B \codot B')_n = B \codot B'_n$.

\begin{prop}
Let $B$ be a complete Hausdorff $k$-$k'$-biring. The functor
\[B \codot - \colon \CAlg_{k'} \to \CAlg_k\] is left adjoint to $\Spf_k(B) \colon \CAlg_k \to \CAlg_{k'}$. \qed
\end{prop}

\begin{prop}
	For a complete Hausdorff $k$-$k'$-biring $B$ and complete Hausdorff $k'$-algebra $A$, the formal scheme $\Spf_k(B \codot A)$ is given by the composition
	$$ \CAlg_k \xrightarrow{\Spf_k(B)} \CAlg_{k'} \xrightarrow{\Spf_{k'}(A)} \Set.$$
	Hence, $\codot$ lifts to a functor
\[\CBiring_{k,k'} \times \CBiring_{k',k''} \to \CBiring_{k,k''}
\] and $(\CBiring_{k,k}, \codot, \mathcal{I})$ forms a monoidal category. \qed
\end{prop}

Proceeding as in the discrete case, we can now define structure which precisely models composition of operations.

\begin{df}
	We define the category of \emph{complete Hausdorff $k$-plethories} $\CPleth_k$ to be the category of monoids in $\CBiring_{k,k}$. Explicitly, a \emph{complete Hausdorff $k$-plethory} is a complete Hausdorff $k$-$k$-biring $P$ together with two additional complete Hausdorff biring morphisms
	\begin{align}
	\circ & \colon P \codot P \to P \tag{composition} \\
	u & \colon \mathcal{I} \to P \tag{identity}
	\end{align}
	satisfying the usual relations for a monoid.
\end{df}

\begin{ex}
	The initial complete Hausdorff $k$-plethory is the complete Hausdorff $k$-$k$-biring $\mathcal{I}$  of Example~\ref{ex:id} together with the canonical structure maps.
\end{ex}

\begin{ex}
	The complete Hausdorff $k$-$k$-biring $\Set(k, k)$ together with composition of maps and the identity map forms an ungraded complete Hausdorff $k$-plethory. We use $\iota$ to denote the identity on $k$ and $1$ to denote the constant map
$k\to k$ sending $\kappa$ to $1$ for all $\kappa\in k$.
\end{ex}

\begin{ex}
As detailed in \cite{BorgerW:2005}, in the discrete setting we have a free functor from the category of $k$-$k$-birings to the category of $k$-plethories, analogous to the tensor algebra construction over a $k$-module. In the topological setting, we define $T_{\codot}(B)$, the free complete Hausdorff $k$-plethory over a complete Hausdorff $k$-$k$-biring $B$ by
$$T_{\codot}(B) = \widehat{\bigotimes_{n \geq 0}} B^{\codot n} $$
together with the obvious identity and composition.
\end{ex}

We wish to encode not only the composition of operations, but the actions of operations on suitable algebras. This leads to a result which proves useful for calculations.

\begin{df}
	For a complete Hausdorff $k$-plethory $P$, we define the category of \emph{complete Hausdorff $P$-algebras} to be the category of algebras over the monad $P \codot - \colon \CAlg_k \to \CAlg_k$. We write $r(x)$ for the image of $r \odot x$ under the action map $P \codot A \to A$.
\end{df}

\begin{ex}
For a space $X$,
	the degree zero complex $K$-theory, $K(X) = [X, \Z \times BU]$, admits the structure of a $\Set(\Z, \Z)$-algebra. The action of $f\in \Set(\Z, \Z)$ sends the class of $x:X\to\Z\times BU$ to the class of the composite
	$$X \xrightarrow{x} \Z \times BU \xrightarrow{f \times 1} \Z \times BU.$$
\end{ex}

\begin{prop}
	\label{prop:coalg-from-action}
	For a complete Hausdorff $k$-plethory $P$, the structure maps are complete Hausdorff $P$-algebra maps and so the co-algebraic structure is determined by the action on complete Hausdorff $P$-algebras. For example, if $r \in P$ then $r(xy) = r_{[1]}(x)r_{[2]}(y)$ for all $x,y$ in any complete Hausdorff $P$-algebra $A$ if and only if $\Delta^\times r = r_{[1]} \otimes r_{[2]}$.
\end{prop}
\begin{proof}
See~\cite[Section 4]{TallW:1970} for the discrete case, which generalises without difficulty.
\end{proof}

\medskip

We can now give a direct proof of a key result of Stacey and Whitehouse \cite[Corollary 5.4]{StaceyW:2009}. The original proof is an application of a very abstract, but more general result.  For a space $X$, we write $\cE^*(X)$ for the completed $E$-cohomology of $X$ with respect to the skeletal filtration.

\begin{thm}
	\label{thm:cohom-is-pleth}
	Let $E^*(-)$ be a multiplicative cohomology theory. If $E_*(\rE_n)$ is a free $E^*$-module for each $n \in \mathbb{Z}$ then $E^*(\rE_\bullet)$ is a complete Hausdorff $E^*$-plethory. Moreover, for any space $X$ the completed cohomology $\cE^*(X)$ is naturally a $E^*(\rE_\bullet)$-algebra.
\end{thm}
\begin{proof} Since each $E_*(\rE_n)$ is a free $E^*$-module, we have suitable K\"unneth isomorphisms and thus the $E^*$-algebra object structure on $(E_n)_{n \in \Z}$ induces a co-$E^*$-algebra structure on the collection of complete Hausdorff $E^*$-algebras $E^*(\rE_n)$. Thus, $E^*(\rE_\bullet)$ is a complete Hausdorff $E^*$-$E^*$-biring. We define a composition $\circ \colon E^*(\rE_\bullet) \codot E^*(\rE_\bullet) \to E^*(\rE_\bullet)$ by $r \circ s = s^*(r)$ and the unit $u \colon \mathcal{I} \to E^*(\rE_\bullet)$ by $u(\iota_n) = \iota_n \in E^*(\rE_n)$, the universal class. These maps make $E^*(\rE_\bullet)$ a complete Hausdorff $E^*$-plethory by construction.
\end{proof}

As this theory is set up to work with completed cohomology algebras, with respect to the skeletal filtration, we lose some information. In general, the completion of a cohomology algebra contains strictly less information than the uncompleted algebra. In forming the completion, we take the quotient by the \emph{phantom classes}: those which are zero on any finite subcomplex.
In~\cite{Bauer:2014}, Bauer shows that we can avoid this issue by working with the entire pro-system of cohomology algebras. However in many cases of interest, there are results that preclude the existence of phantom classes and thus $E^*(X) = \cE^*(X)$.

Our main results relate to integral complex $K$-theory, so we do not make use of other completions, such as $p$-adic completion, or $I$-adic completion with respect to
an ideal.
\medskip

We introduce some theory of \emph{non-(co-unital)} birings which will prove useful. For brevity, we focus on the discrete, ungraded case but remark that these constructions generalise without difficulty to the topologised, graded setting.

\begin{df}
We define the \emph{non-(co-unital) composition product} $ B \uodot A$ of a non-(co-unital) $k$-$k'$-biring and a non-unital $k'$-algebra $A$ to be the free unital $k$-algebra on the symbols $b \uodot a$, for $b\in B$, $a\in A$, quotiented by the relations enforcing that $b \mapsto b \uodot a$ is a $k$-algebra map together with the relations
\begin{align*}
b \uodot (a_1 + a_2) &= (b_{(1)} \uodot a_1) (b_{(2)} \uodot a_2) \\
b \uodot (a_1 a_2) &= (b_{[1]} \uodot a_1) (b_{[2]} \uodot a_2)  \\
b \uodot (\kappa a) &= \gamma(\kappa)(b) \uodot a \\
b \uodot 0 &= \eps^+(b)
\end{align*}
for all $a,a_1, a_2 \in A$, $b \in B$ and $\kappa \in k'$.
\end{df}

\begin{prop}
	If $B$ is a non-(co-unital) $k$-$k'$-biring, the functor $B \uodot - \colon \Alg_{k'}^! \to \Alg_k$ is left adjoint to $\Spec_k(B) \colon \Alg_k \to \Alg_{k'}^!$.
\end{prop}
\begin{proof} This is the same argument as in the co-unital setting. \end{proof}

For $k$-$k'$-birings $B, B'$ and $k'$-algebras $A, A'$ we have natural isomorphisms
\begin{align*}
B \odot (A \otimes A') &\isom (B \odot A) \otimes (B \odot A'),\\
(B \otimes B') \odot A &\isom (B \odot A) \otimes (B' \odot A),\\
k \odot B &\isom k \isom B \odot k'.
\end{align*}
 These have analogues in the non-(co-unital) setting.

Let $R, S$ be non-unital $k$-algebras. Recall the coproduct $R \ucoprod S$ is given by the $k$-module $ R \oplus S \oplus (R \otimes S)$ together with
multiplication specifed by
 the product of $r_1 + s_1 + r'_1\otimes s'_1$ and $r_2+s_2+r'_2\otimes s'_2$ being given by
$$r_1 r_2 + s_1 s_2 + r_1 \otimes s_2 + r_2\otimes s_1 + r_1 r'_2 \otimes s'_2+ r'_1 r_2 \otimes s'_1 + r'_2\otimes s_1 s'_2 + r'_1\otimes s'_1 s_2 + r'_1 r'_2 \otimes s'_1 s'_2.$$

\begin{prop}
	Let $B$ be a non-(co-unital) $k$-$k'$-biring and $A, A'$ non-unital $k'$-algebras.	We have isomorphisms $B \uodot (A \ucoprod A') \isom (B \uodot A) \otimes (B \uodot A') $ and $k \uodot A \isom k$.
\end{prop}
\begin{proof} For any $k$-algebra $X$ we have isomorphisms
\begin{align*}
\Alg_k(B \uodot (A \ucoprod A'), X) &\isom \Alg^!_{k'}(A \ucoprod A', \Spec_k(B)(X)) \\
&\isom \Alg^!_{k'}(A, \Spec_k(B)(X)) \times \Alg^!_{k'}(A', \Spec_k(B)(X)) \\
&\isom \Alg_k(B \uodot A, X) \times \Alg_k(B \uodot A', X) \\
&\isom \Alg_k((B \uodot A) \otimes (B \uodot A'), X).
\end{align*}
As in Example~\ref{ex:zero}, $k$  is the initial $k$-$k'$-biring corresponding to the constant functor at the zero ring
and the isomorphism $k \uodot A \isom k$ is trivial.
\end{proof}

\begin{prop}
Suppose $B$ is a non-(co-unital) $k$-$k'$-biring and a co-$B'$-module where $B'$ is a $k$-$k'$-biring.
For an augmented $k'$-algebra $A$, we have an isomorphism
\[(B' \otimes B) \odot A \cong (B' \odot A) \otimes (B \uodot IA),
\] where $IA$ denotes the augmentation ideal of $A$.
\end{prop}
\begin{proof} For any $k$-algebra $X$ we have isomorphisms
\begin{align*}
\Alg_k((B' \otimes B) \odot A, X) &\isom \Alg_{k'}(A, \Spec_k(B' \otimes B)(X)) \\
&\isom \Alg_{k'}(A, \Spec_k(B')(X) \oplus \Spec_k(B)(X)) \\
&\isom \Alg_{k'}(A, \Spec_k(B')(X)) \times \Alg_{k'}^!(IA, \Spec_k(B)(X)) \\
&\isom \Alg_k(B' \odot A, X) \times \Alg_k(B \uodot IA, X) \\
&\isom \Alg_k((B' \odot A) \otimes (B \uodot IA), X).\qedhere
\end{align*}
\end{proof}

We write $\cuodot$ for the non-(co-unital) composition product in the completed setting.

\section{Ungraded $K$-theory operations}
\label{sec:ungddops}

The study of the operations of ungraded $K$-theory is a classical subject in algebraic topology~\cite{Atiyah:1967}
and it is well known that the degree zero $K$-cohomology, $K(X)=K^0(X)$, of a space $X$ naturally forms a $\lambda$-ring. In this section we exhibit a concise description of the operations in a plethystic setting.

The classifying space $BU$ of the infinite unitary group is central to the study of $K$-theory and admits the structure of a non-unital ring space, with abelian group structure corresponding to the direct sum of vector bundles, and (non-unital) multiplication induced by the tensor product. Thus, since $K(BU)$ is free as a $\Z$-module, $K(BU)$ naturally admits the structure of a non-(co-unital) complete Hausdorff $\Z$-$\Z$-biring by the ungraded and non-unital analogue of \cref{thm:cohom-is-pleth}.

\begin{thm}
We have an isomorphism of non-(co-unital) complete Hausdorff $\Z$-$\Z$-birings
\begin{equation}
K(BU) \isom \Z[[\lambda^1 \iota, \lambda^2 \iota, \dots]]
\end{equation}
where $\iota$ is represented by the inclusion $BU \simeq \{0\} \times BU \subseteq \Z \times BU$. The filtration ideals are given by the kernels of the projection maps $ Z[[\lambda^1 \iota, \lambda^2 \iota, \dots ]] \to  Z[[\lambda^1 \iota, \dots , \lambda^n \iota]]$, and the non-(co-unital) biring structure is determined by
\begin{align*}
\Delta^+(\lambda^k \iota) &= \sum_{i+j=k} \lambda^i \iota \otimes \lambda^j \iota \\
\Delta^\times(\lambda^k \iota) &= P_k(\lambda^1  \iota \otimes 1, \dots, \lambda^k \iota \otimes 1 ; 1 \otimes \lambda^1 \iota, \dots, 1 \otimes \lambda^k \iota),
\end{align*}
where the $P_k$ are the universal polynomials arising in the theory of $\lambda$-rings, see~\cite[Definition 1.10]{Yau:2010}.
\end{thm}
\begin{proof} The description of $K(BU)$ as a power series ring in the lambda operations is well-known
and the remaining structure follows directly from the theory of  $\lambda$-rings.
\end{proof}

Since $\Z \times BU$ is the representing space for ungraded $K$-theory, studying the operations corresponds to understanding the complete Hausdorff $\Z$-plethory $K(\Z \times BU)$.

\begin{prop}\label{isobirings}
	We have an isomorphism of (ungraded) complete Hausdorff $\Z$-$\Z$-birings,
	$$ K(\Z \times BU) \isom \Set(\Z, \Z) \cotimes K(BU) $$
	where the $\Z$-$\Z$-biring structure is specified in \cref{ex:biring-adjoin-unit}.
\end{prop}
\begin{proof} By the K\"unneth theorem, we have an isomorphism of rings.
We write \[\theta: \Set(\Z, \Z) \cotimes K(BU)\to K(\Z \times BU)\] for this isomorphism.
 Since the abelian group structure on $\Z \times BU$ is given by the product structure, this is an isomorphism of Hopf algebras. It remains to show that $\theta$ respects the co-multiplicative structure. The element $\chi_d \otimes x \in \Set(\Z, \Z) \cotimes K(BU)$ corresponds to $\pi_1^*\chi_d \pi_2^*x$ under the K\"unneth isomorphism $\theta$, where $\pi_1, \pi_2$ denote the canonical projections. By \cref{prop:coalg-from-action}, we can compute the comultiplication by considering the action of $\pi_1^*\chi_d \pi_2^*x$ on general $\alpha, \beta \in K(X)$. Assume that $X$ is connected and thus has a unique up to homotopy choice of base point. Denote the map induced by the inclusion of the base point by $\eps \colon K(X) \to \Z$. The case of general $X$  will follow by considering each connected component individually. For $f \in \Set(\Z, \Z)$ and $x \in K(BU)$, we have $\pi_1^*f(\alpha) = f(\eps(\alpha))$ and $\pi_2^*x(\alpha) = x(\alpha - \eps(\alpha))$. In $K(X)$,
\begin{align*}
	(\pi_1^*\chi_d &\pi_2^*x)(\alpha \beta)
	\\
	 &= \chi_d(\eps(\alpha)\eps(\beta)) x\left[\alpha \beta - \eps(\alpha)\eps(\beta)\right]  \\
	&= \sum_{rs=d} \chi_r(\eps(\alpha)) \chi_s(\eps(\beta)) x\left[(\alpha  - \eps(\alpha)(\beta - \eps(\beta)) + \eps(\alpha)(\beta - \eps(\beta)) + \eps(\beta)(\alpha - \eps(\alpha))\right] 	\\
	&= \sum_{rs = d} \chi_r(\eps(\alpha)) \chi_s(\eps(\beta))
	\pi_2^*\left[ x_{(1)[1]} \gamma(\eps(\beta))(x_{(3)}) \right](\alpha)  \pi_2^*\left[ x_{(1)[2]} \gamma(\eps(\alpha))(x_{(2)}) \right](\beta) \\
	&= \sum_{rs = d} \chi_r(\eps(\alpha)) \chi_s(\eps(\beta))
	\pi_2^*\left[x_{(1)[1]} \gamma(s)(x_{(3)}) \right](\alpha)  \pi_2^*\left[ x_{(1)[2]} \gamma(r)(x_{(2)}) \right](\beta) \\
	&= \sum_{rs = d} \left(\pi_1^* \chi_r \pi_2^*\left[ x_{(1)[1]} \gamma(s)(x_{(3)}) \right] \right)(\alpha) \left( \pi_1^* \chi_s \pi_2^*\left[ x_{(1)[2]} \gamma(r)(x_{(2)}) \right] \right)(\beta)
\end{align*}
where the fourth equality follows since $\chi_i(j) = \delta_{ij}$, the Kronecker delta. Hence
$$ \Delta^\times (\pi_1^*\chi_d \pi_2^*x) = \sum_{rs = d} \pi_1^* \chi_r \pi_2^*\left[ x_{(1)[1]} (\gamma(s)(x_{(3)}))\right]  \otimes  \pi_1^* \chi_s \pi_2^*\left[ x_{(1)[2]} (\gamma(r)(x_{(2)}))\right]. $$
Therefore the K\"unneth isomorphism respects the comultiplication $\Delta^\times$. To see that  the co-unit is preserved, notice that $(\pi_1^*f \pi_2^*x)(1) = f(1)x(0) = \eps^\times(f)\eps^+(x)$. \end{proof}

Recall that for a based space $X$, the reduced $K$-theory, which we denote $K(X, o)$, is the kernel
of the augmentation given by the map induced by the inclusion of the basepoint.

\begin{prop}
We have a map of rings $K(BU) \cuodot K(BU, o) \to K(BU)$ determined by
$$ \lambda^i \iota \circ \lambda^j \iota = P_{i,j}(\lambda^1 \iota, \dots, \lambda^{ij} \iota), $$
where the $P_{i,j}$ are the universal polynomials arising in the theory of $\lambda$-rings, see~\cite[Definition 1.10]{Yau:2010}.
\end{prop}
\begin{proof} This is immediate from the properties of $\lambda$-rings. \end{proof}

For based spaces $X, Y$, the  cohomological K\"unneth isomorphism induces an isomorphism of non-unital rings on reduced cohomology
$$ K(X \times Y, o) \isom K(X, o) \cucoprod K(Y,o ). $$

Recall that the co-zero map, which defines the augmentation ideal, on $\Set(\Z, \Z)$ is given by the evaluation map $\eps^+ :  \Set(\Z, \Z)\to \Z$, with $\eps^+(f) = f(0)$. We have an isomorphism
$$ I\left(\Set(\Z,\Z) \cotimes K(BU) \right) \isom I\Set(\Z,\Z) \cucoprod K(BU, o). $$

We now \emph{define} the appropriate composition on $\Set(\Z, \Z) \cotimes K(BU)$ by the following sequence of maps, where $\phi_R \colon \Set(\Z, \Z) \codot \Z \to \Z$ and $\phi_L \colon \Z \cuodot I\Set(\Z, \Z) \to \Z$  denote the canonical isomorphisms.
\begin{center}
\scalebox{.9}{
\begin{tikzcd}
\left(\Set(\Z, \Z) \cotimes K(BU) \right) \codot \left(\Set(\Z, \Z) \cotimes K(BU) \right)
\arrow{d}{\isom}
\\
\left( \Set(\Z, \Z) \codot \Set(\Z, \Z) \right) \cotimes \left( \Set(\Z, \Z) \codot K(BU) \right) \cotimes \left( K(BU) \cuodot I\Set(\Z, \Z) \right) \cotimes \left( K(BU) \cuodot K(BU, o) \right)
\arrow{d}{1 \cotimes 1 \codot \eps^+ \cotimes \eps^+ \cuodot 1 \cotimes 1}
\\
\left( \Set(\Z, \Z) \codot \Set(\Z, \Z) \right) \cotimes \left( \Set(\Z, \Z) \codot \Z \right) \cotimes \left( \Z \cuodot I\Set(\Z, \Z) \right) \cotimes \left( K(BU) \cuodot K(BU, o) \right)
\arrow{d}{\circ \cotimes \phi_R \cotimes \phi_L \cotimes \circ}
\\
\Set(\Z, \Z) \cotimes \Z \cotimes \Z \cotimes K(BU)
\arrow{d}{\isom}
\\
\Set(\Z, \Z) \cotimes K(BU)
\end{tikzcd}
}
\end{center}

On the level of elements, for $d\in\Z$, $g\in\Set(\Z,\Z)$, $x,y \in K(BU)$, this reads as
$$ (\chi_d \otimes x) \circ (g \otimes y) = \sum_{rs=d} \chi_r(\eps^+(y)) \chi_s \circ g \otimes \gamma(s)(x) \circ \left(y - \eps^+ y \right),  $$
with identity given by $1 \otimes \lambda^1\iota + \iota \otimes 1$. Note that composition respects sums on  the left so it is enough to specify it on the above elements.

\begin{thm}
	\label{thm:ungd-pleth-isom}
	We have an isomorphism of ungraded complete Hausdorff $\Z$-plethories
	$$ K(\Z \times BU) \isom \Set(\Z, \Z) \cotimes K(BU). $$
\end{thm}
\begin{proof}
By Proposition~\ref{isobirings}, we have an isomorphism of birings and it remains to check compatibility with composition.
Let $d\in \Z$, $g \in \Set(\Z, \Z)$, $x,y \in K(BU)$ and $\alpha \in K(X)$. We have
\begin{align*}
	\theta(\chi_d \otimes x) \circ \theta(g \otimes y)
	&=(\pi_1^* \chi_d \pi_2^* x) \circ (\pi_1^* g \pi_2^* y)(\alpha)\\
	&= (\pi_1^* \chi_d \pi_2^* x)(g(\eps(\alpha)) y(\alpha - \eps(\alpha))) \\
	&=(\pi_1^* \chi_d)(g(\eps(\alpha)) y(\alpha - \eps(\alpha))) (\pi_2^* x)(g(\eps(\alpha)) y(\alpha - \eps(\alpha))) \\
	&= \left[\chi_d(g(\eps(\alpha)) \eps^+(y)) \right] \left[ \gamma(g(\eps(\alpha)))(x) \circ \left(y- \eps^+(y) \right)(\alpha - \eps(\alpha))\right] \\
	&= \sum_{rs=d}\left[\chi_r(\eps^+(y)) \chi_s(g(\eps(\alpha))) \right] \left[ \gamma(g(\eps(\alpha)))(x) \circ \left(y- \eps^+(y) \right)(\alpha - \eps(\alpha))\right] \\
	&= \sum_{rs=d}\left[\chi_r(\eps^+(y)) \chi_s(g(\eps(\alpha))) \right] \left[ \gamma(s)(x) \circ \left(y- \eps^+(y) \right)(\alpha - \eps(\alpha))\right] \\
	&= \sum_{rs=d} \pi_1^*\left[\chi_r(\eps^+(y)) \chi_s \circ g \right] \pi_2^*\left[ \gamma(s)(x) \circ \left(y- \eps^+(y) \right)\right](\alpha)\\
	&=\theta((\chi_d \otimes x) \circ (g \otimes y)).
\end{align*}

Finally, we note that $(\pi_1^* 1 \pi_2^* \lambda^1\iota + \pi_1^* \iota \pi_2^* 1)(\alpha) = \alpha - \eps(\alpha) + \eps(\alpha) = \alpha$.
\end{proof}

\section{Plethories with looping}
\label{sec:looping}

The standard definition~\cite{Boardman:1995} of a (graded) generalised cohomology theory is a $\Z$-graded collection of well-behaved functors $E^n(-) \colon \Ho \to \Ab$ together with \emph{suspension isomorphisms}.
For a based space $X$, the corresponding reduced cohomology groups are denoted $E^n(X, o)$ and are defined
as the kernel of the map induced by inclusion of the base point, as we already saw in the case of $K$-theory.
The theory is extended to pairs, by defining the cohomology of a pair to be the reduced cohomology of the quotient
space.
The suspension isomorphisms can be expressed as isomorphisms of abelian groups 
\[\Sigma \colon E^n(X) \to E^{n+1}(S^1 \times X, o \times X)\]
 for all spaces $X$ and all $n \in \Z$,
or equivalently 
\[\Sigma \colon E^n(X, o) \isom E^{n+1}(\Sigma X, o)\]
 on reduced cohomology groups where $\Sigma X = S^1 \wedge X$ denotes the reduced suspension.

The suspension isomorphisms impose additional structure on the algebras over a plethory of unstable cohomology operations. Since plethories are \emph{precisely} the structure which acts on algebras, we will need extra structure to encode this information.

Recall that for a based operation $r \colon E^n(-) \mapsto E^m(-)$, there is the \emph{looped operation} $\Omega r \colon E^{n-1}(-) \mapsto E^{m-1}(-)$ defined by the following commutative diagram.

	\begin{center}
				\begin{tikzcd}
					E^{n-1}(X) \arrow{r}{\Sigma} \arrow{d}{\Omega r} & E^n(S^1 \times X, o \times X) \arrow{d}{r} \\
					E^{m-1}(X) \arrow{r}{\Sigma} & E^m(S^1 \times X, o \times X)
				\end{tikzcd}
	\end{center}
	
\begin{df}
Let $P$ be a complete Hausdorff $k$-plethory. We define the \emph{augmentation ideal} $IP$, \emph{primitives} $\Prim(P)$ and \emph{indecomposables} $\Indecomp(P)$ by
\begin{align*}
IP &= \ker \epsilon^+, \\
\Prim(P) &= \{x \in P \st \Delta^+(x) = 1 \otimes x + x \otimes 1\}, \\
\Indecomp(P) &= \frac{IP}{(IP)^2}.
\end{align*}
The additional structure of a plethory induces additional structure on these familiar constructions from Hopf algebra theory as detailed in the ungraded setting in \cite{BorgerW:2005} and the graded setting in \cite{Mycroft:2017}. These constructions carry over to the topological context without difficulty.
\end{df}

\begin{df}
	\label{df:looping}
	We define a \emph{$k$-plethory with looping} to be a complete Hausdorff $k$-plethory $P$ equipped with a continuous bidegree $(-1,-1)$ $k$-module map $\Omega \colon IP \to IP$ satisfying the following properties.
	\begin{enumerate}
		\item $\Omega$ is zero on $(IP)^2$ and takes values in primitives.
		That is, $\Omega$ factors as $ IP \xrightarrow{\pi} \Indecomp(P) \to \Prim(P) \subseteq IP$, where $\pi$ denotes the quotient map.
		\item For $r \in IP$, $\Delta^\times (\Omega r) = (-1)^{\deg_*(r_{[1]})} \sigma^{\deg_\bullet(r_{[1]})} r_{[1]} \otimes \Omega r_{[2]}$.
		\item For $r,s \in IP$, $\Omega(r \circ s) = \Omega r \circ \Omega s$.
		\item For all $n \in Z$, $\Omega(\iota_n) = \iota_{n-1}$.
	\end{enumerate}

	A map of plethories $f \colon P \to P'$ is a map of $k$-plethories with looping if $\Omega f(r) = f \Omega(r)$ for all $r \in P$. We denote the category of $k$-plethories with looping by $\LoopPleth_k$.
\end{df}

\begin{thm}
	Let $E^*(-)$ be a graded cohomology theory. If $E_*(\rE_n)$ is a free $E^*$-module for each $n \in \mathbb{Z}$ then $E^*(\rE_\bullet)$ is an $E^*$-plethory with looping.
\end{thm}
\begin{proof} Looping of operations is defined for based maps and so gives a map from $IE^*(\rE_\bullet)$ to $IE^*(\rE_\bullet)$, of bidegree $(-1,-1)$.
 It suffices to show that it satisfies properties (1) to (4) of \cref{df:looping}. For property (1),
 see~\cite[Corollary 2.18]{BoardmanJ:1995}. Properties (3) and (4) are immediate from the definition. For (2), let $x, y \in \cE^*(X)$ for some space $X$ and let $\pi_2 \colon S^1 \times X \to X$ denote the canonical projection.
To determine the comultiplication of a looped operation, we consider the action on products.
By definition we have
\begin{align*}
\Sigma (\Omega r)(xy) &= r(\Sigma (xy)) \\
&= r\left((-1)^{|x|} (\pi_2^*x) \Sigma y \right) \\
&= r_{[1]}\left(\pi_2^*((-1)^{|x|}x)\right) r_{[2]}(\Sigma y) \\
&= \pi_2^*\left((\sigma^{|x|} r_{[1]})(x)\right) \Sigma (\Omega r_{[2]})(y) \\
&= \Sigma \left( (-1)^{\deg_*(r_{[1]})}  (\sigma^{|x|} r_{[1]})(x) (\Omega r_{[2]})(y)  \right)
\end{align*}
and thus $$\Delta^\times (\Omega r) = (-1)^{\deg_*(r_{[1]})}  \sigma^{\deg_\bullet(r_{[1]})} r_{[1]} \otimes \Omega r_{[2]}.$$
\end{proof}

\begin{df}
An \emph{ideal} of a $k$-plethory with looping is an ideal $\mathcal{J}$ of a $k$-plethory such that $\Omega x \in \mathcal{J}$ for all $x \in \mathcal{J}$.
\end{df}

It is immediate that if $\mathcal{J} \subseteq P$ is an ideal of a $k$-plethory with looping then the canonical map $P \to P/\mathcal{J}$ is a map of $k$-plethories with looping.

In many settings, we obtain interesting collections of operations by considering loopings and composites of a small set of operations.

\begin{df}
Let $P$ be a complete Hausdorff $k$-plethory. We define the complete Hausdorff $k$-$k$-biring $P_\Omega$ to be the free $k$-algebra generated by the symbols $\Omega^0 x$ for $x \in P$ together with $\Omega^l x$ for $x \in IP$ and $l > 0$, quotiented by the ideal generated by the relations
\begin{align*}
	\Omega^0( x + y) &= \Omega^0(x) + \Omega^0(y) \\
	\Omega^0(xy) &= (\Omega^0 x)(\Omega^0 y) \\
	\Omega^0(\kappa) &= \kappa, \text{for $\kappa\in k$}\\
	\Omega^l(x + y) &= \Omega^l(x) + \Omega^l(y) \\
	\Omega^l(xy) &= \eps^+(x) \Omega^l(y) + \eps^+(y) \Omega^l(x).
\end{align*}
The bigrading is determined by $\deg_*(\Omega^k x) = \deg_*(x) - k$ and $\deg_\bullet(\Omega^k x) = \deg_\bullet(x) - k$. The identification $x \mapsto \Omega^0 x \in P_\Omega$ yields a canonical $k$-algebra map $P \hookrightarrow P_\Omega$. The biring structure on $P_\Omega$ is given by defining the elements $\Omega^k x$ to be primitive for $k > 0$, the canonical map $P \to P_\Omega$ to be a monomorphism of $k$-$k$-birings together with the following formulae for $k > 0$.
\begin{align*}
	\Delta^\times(\Omega^k x) &= (-1)^{k \deg_*(x_{[1]})} \sigma^{k\deg_\bullet(x_{[1]})} x_{[1]} \otimes \Omega^k x_{[2]} \\
	\eps^\times(\Omega^k x) &= (-1)^{k \deg_*(x_{[1]})} \eps^\times \left( \sigma^{k \deg_\bullet(x_{[1]})} x_{[1]}\right)  \\
	\beta\lambda(\Omega^k x) &= (\beta \lambda)(x_{[1]}) \eps^\times(\Omega^k x_{[2]})
\end{align*}
We define $\Omega P$, the \emph{free $k$-plethory with looping on $P$} to be the complete Hausdorff $k$-plethory $T_{\codot}(P_\Omega)$ quotiented by the relations
\begin{align*}
\Omega^k x \circ \Omega^k y &= \Omega^k(x \circ y) \\
\Omega^k \iota_n &= \iota_{n-k}.
\end{align*}
The looping in $\Omega P$ is given by $\Omega(\Omega^k x) = \Omega^{k+1} x$ and a map $f \colon P \to P'$ of complete Hausdorff $k$-plethories induces a map of $k$-plethories with looping $\Omega P \to \Omega P'$ by $f(\Omega^k x) = \Omega^k f(x)$. This construction defines a functor $\Omega: \CPleth_k\to \LoopPleth_k$.
\end{df}

\begin{prop}
	\label{prop:looping-adjunction}
	The forgetful functor $U \colon \LoopPleth_k \to \CPleth_k$ is right adjoint to $\Omega$.
\end{prop}
\begin{proof} A map of complete Hausdorff $k$-plethories $f \colon P \to UP'$ defines a map of $k$-plethories with looping $\hat f \colon \Omega P \to P'$ by $\hat f(\Omega^k x) = \Omega^k(f(x))$. Conversely, a map of $k$-plethories with looping $\Omega P \to P'$ restricts to a map of complete Hausdorff $k$-$k$-plethories $P \to UP'$ via the canonical inclusion $P \to \Omega P$. \end{proof}

\section{$K$-theory operations as a free plethory with looping}
\label{sec:mainresult}

We briefly study the $K$-theory operations of odd source degree. Since complex $K$-theory is represented in odd degrees by the infinite unitary group $U$, this is tantamount to understanding the Hopf algebra $K^*(U)$. We then relate these results to the $\lambda$-operations and show that in a suitable context, the $\lambda$-operations generate all $K$-theory operations.

Write $\Lambda^k \colon U(n) \to U{n \choose k} \subseteq U$ for the exterior power representation of the unitary group and let $\mu^k_n \in K^{-1}(U(n))$ denote the class represented by $\Lambda^k$.

\begin{thm}[{\cite[Theorem 2.7.17]{Atiyah:1967}}]
	\label{thm:k-theory-unitary-group}
	We have an isomorphism of $K^*$-algebras
	$$ K^*(U(n)) \isom \Lambda_{K^*}[\mu^1_n, \dots, \mu^n_n].$$
	Moreover, if $i \colon U(n-1) \to U(n)$ denotes the standard inclusion map then $ i^*(\mu_n^k) = \mu_{n-1}^k + \mu_{n-1}^{k-1}.$
\end{thm}

We remark that the choice of degree for the elements $\mu^k_n \in K^*(U(n))$ is arbitrary and we could choose any odd degree. Our selection is motivated by a relation to the even degree operations: the looping of the $\lambda$-operations will be expressible in terms of the $\mu^k_n$ and we chose the $\lambda$-operations to lie in cohomological degree zero.

To understand the relationship between the $\mu_n^k$ and our choice of generators of $K(BU)$ it proves fruitful to understand the representing maps of the $\lambda$-operations. By a classical result of Anderson \cite{Anderson:1983}, there are no phantom operations in $K$-theory and thus $K(BU) \isom \varprojlim\limits_n K(BU(n))$. Let $\beta^k_n \in K(BU(n), o)$ be represented by
\[
B \Lambda^k \colon BU(n) \to BU \simeq \{0\} \times BU \subseteq \Z\times BU.
\]

\begin{prop}
	\label{prop:rep-lambda-ops}
	Define $\lambda^k_n = \sum_{i=0}^{k} {-n \choose i}\beta^{k-i}_n \in K(BU(n), o)$. The following hold.
	\begin{enumerate}
		\item For $j=Bi \colon BU(n) \to BU(n+1)$, we have $j^* \lambda^k_{n+1} = \lambda^k_n$.
		\item The element $\lambda^k \iota \in K(BU, o) \isom \varprojlim\limits_n K(BU(n), o)$ corresponds to the inverse limit of the $\lambda^k_n \in K(BU(n), o) $.
	\end{enumerate}
\end{prop}
\begin{proof} The first result follows immediately since $j^* \beta^k_{n+1} = \beta^k_{n} + \beta^{k-1}_{n}$. For the second, let $X$ be a compact Hausdorff space, so the representing
map for $x\in K(X, o)$ factors via $\Z \times BU(n)$ for some $n$.
Let $x = [\xi] - n \in K(X, o)$.  Now the composition
$$ X \xrightarrow{x} \Z \times BU(n) \xrightarrow{\Z \times \lambda^k_n} \Z \times BU  $$
represents the virtual bundle 
$$\sum_{i=0}^{k} {-n \choose i}\left[\Lambda^{k-i}[\xi] -  {n \choose k-i}\right] = \sum_{i=0}^{k} {-n \choose i} \Lambda^{k-i}[\xi] = (\lambda^k \iota)(x).$$ \end{proof}

This linear combination of generators allows us to compute $K^*(U) = \varprojlim K^*(U(n))$ in a form closely related to our description of $K(BU)$.
\begin{prop}\label{prop:odd-gens}
	Let $l_n^k = \sum \limits_{i=0}^{k-1} {-n \choose i} \mu^{k-i}_n \in K^{-1}(U(n))$ for $k \leq n$.
	\begin{enumerate}
		\item If $i \colon U(n-1) \to U(n)$ is the inclusion map as above then $ i^*(l_n^k) = l_{n-1}^k. $
		\item We have an isomorphism of $K^*$-algebras $K^*(U(n)) \isom \Lambda_{K^*}[l^1_n, \dots, l^n_n].$
		\item We have an isomorphism of $K^*$-algebras $K^*(U) \isom \Lambda_{K^*}[l^1, l^2, \dots]$ where if $\iota \colon U(n) \to U$ denotes the inclusion then $\iota^*l^k = l^k_n$.
	\end{enumerate}
\end{prop}
\begin{proof} 
This follows directly from \cref{thm:k-theory-unitary-group}.
\end{proof}

The following result is now an immediate consequence and the motivation for the definition of the odd degree operations $l^k$.

\begin{cor}
	\label{thm:loop-lambda}
	The composition
	$$ \Set(\Z, \Z) \cotimes K(BU) \xrightarrow{\theta} K(\Z \times BU) = K^0(\Z \times BU) \xrightarrow{\Omega} K^{-1}(U), $$
	is determined by $f \otimes \lambda^k \iota \mapsto f(0) l^k$, for $f\in\Set(\Z,\Z)$.
\end{cor}
\begin{proof}
Since $\Omega(\Z \times BU) = \Omega(\{0\} \times BU)$, it suffices to consider the restriction of $\pi_1^* f \pi_2^*(\lambda^k \iota)$ to $\{0\} \times BU \simeq BU$ which is $f(0) \lambda^k \iota \in K(BU)$ and so $\Omega (f \otimes \lambda^k \iota) = f(0) \Omega(\lambda^k \iota)$. Now, by \cref{prop:rep-lambda-ops}, $\lambda^k \iota$ is represented by the inverse limit of the maps
\[
	\sum_{i=0}^{k} {-n \choose i} B\Lambda^{k-i} \colon BU(n) \to BU.
\]
 Since $\Omega B \simeq 1$, we see that $\Omega(\lambda^k \iota)$ is represented by the inverse limit of the maps
\[
	\sum_{i=0}^{k} {-n \choose i} \Lambda^{k-i} \colon U(n) \to U
\] and hence $\Omega(\lambda^k \iota) = l^k$.
\end{proof}

The remaining piece of structure to understand is the looping of the odd degree operations.

\begin{df}
\label{df:linear-poly}
Let $P_l \in \Z[x_1, \dots, x_l; y_1, \dots, y_l]$ denote the universal polynomial encoding the action of the $\lambda$-operation $\lambda^l$ on products in a $\lambda$-ring \cite[Definition 1.10]{Yau:2010}. We define the \emph{left-linearisation}, $P^L_l$, of $P_l$ to be the polynomial given by the sum of the monomials of $P_l$ containing a single $x_i$. Concretely, if we define $|x_i| = 1, |y_j| = 0$, for all $i, j$, then $P^L_l$ is the degree $1$ homogeneous part of $P_l$.
\end{df}

\begin{prop}
	\label{prop:loop-l}
	For $l^k \in K^{-1}(U)$, we have
	$$ \Omega l^k = 1 \otimes P^L_k(1, -1, \dots, (-1)^{k-1}; \lambda^1 \iota, \dots,  \lambda^k \iota)
	\quad \in K^{-2}(\Z \times BU). $$
\end{prop}
\begin{proof} By \cref{thm:loop-lambda}, $\Omega l^k = \Omega^2(\pi_1^*f \pi_2^*\lambda^k \iota)$ for any $f$ with $f(0) = 1$. Now let $\alpha \in K(X)$ and denote the degree $2$ suspension element by $u_2 = [\xi_1] - 1 \in K(S^2, o)$ where $\xi_1$ is the canonical line bundle over $S^2 \simeq \C P^1$. Then we have
\begin{align*}
(\Sigma^2 \Omega l^k)(\alpha) &= \Sigma^2 \Omega^2 (\pi_1^*f \pi_2^* \lambda^k \iota)(\alpha) \\
&= (\pi_1^*f \pi_2^* \lambda^k \iota)(u_2 \times \alpha) \\
&= f(\eps(u_2)\eps(\alpha)) \lambda^k (u_2 \times \alpha) \\
&= P_k(\lambda^1(u_2) \times 1, \dots, \lambda^k(u_2)\times 1; 1 \times \lambda^1(\alpha), \dots, 1\times\lambda^k(\alpha)) \\
&= u_2 \times P^L_k(1, -1, \dots, (-1)^{k-1}; \lambda^1(\alpha), \dots, \lambda^k(\alpha)) \\
&= \Sigma^2 P^L_k(1, -1, \dots, (-1)^{k-1} ; \lambda^1\iota, \dots, \lambda^k\iota)(\alpha)
\end{align*}
where the penultimate equality follows since $\lambda^i(u_2) = (-1)^{i-1} u_2$, and $(u_2)^2 = 0$.
\end{proof}

We are now in a position to prove our main result.

\begin{thm}
\label{thm:main}
	We have an isomorphism of $\Z/2$-graded $\Z$-plethories with looping,
	$$ K^*(\rK_\bullet) \cong \frac{\Omega(\Set(\Z, \Z) \cotimes K(BU))}{\mathcal{I}}, $$
	where $\mathcal{I}$ is the plethystic ideal with looping generated by the relations
	\begin{align*}
		\Omega(f \otimes \lambda^p \iota) &= f(0) \Omega(1 \otimes \lambda^p\iota), \\
		\Omega^2(f \otimes \lambda^p \iota) &= f(0) \otimes P^L_p(1, -1, \dots, (-1)^{p-1}; \lambda^1 \iota, \dots,  \lambda^p \iota),
	\end{align*}
	for all $p\geq 1$.
\end{thm}
\begin{proof} From \cref{thm:ungd-pleth-isom} we have seen that we have an isomorphism of complete Hausdorff
$\Z$-plethories
\[\theta \colon \Set(\Z, \Z) \cotimes K(BU) \xrightarrow{\isom} K^0(\rK_0) \subseteq K^*(\rK_\bullet).
\]
By \cref{prop:looping-adjunction} this extends to a map of $\Z$-plethories with looping
\[\Omega(\Set(\Z, \Z) \cotimes K(BU)) \to K^*(\rK_\bullet),
\] which is surjective by \cref{prop:odd-gens} and \cref{thm:loop-lambda}. By \cref{thm:loop-lambda} and \cref{prop:loop-l} the kernel of this map is precisely $\mathcal{I}$. \end{proof}

\newcommand{\noopsort}[1]{} \newcommand{\singleletter}[1]{#1}


\begin{thebibliography}{10}

\bibitem{Anderson:1983}
D.~W. Anderson.
\newblock There are no phantom cohomology operations in ${K}$-theory.
\newblock {\em Pacific J. Math.}, 107(2):279--306, 1983.

\bibitem{Atiyah:1967}
M.~Atiyah.
\newblock {\em K-{T}heory ({L}ecture notes by D. W. Anderson).}
\newblock 1967.

\bibitem{Bauer:2014}
T.~Bauer.
\newblock Formal plethories.
\newblock {\em Adv. Math.}, 254:497--569, 2014.

\bibitem{BergmanH:1996}
G.~M. Bergman and A.~O. Hausknecht.
\newblock {\em Co-groups and co-rings in categories of associative rings},
  volume~45 of {\em Mathematical Surveys and Monographs}.
\newblock American Mathematical Society, Providence, RI, 1996.

\bibitem{Bergman:2015}
G.M. Bergman.
\newblock {\em An Invitation to General Algebra and Universal Constructions}.
\newblock Springer International Publishing, 2015.

\bibitem{Boardman:1995}
J.~M. Boardman.
\newblock Stable operations in generalized cohomology.
\newblock In {\em Handbook of Algebraic Topology}, pages 585--686.
  North-Holland, Amsterdam, 1995.

\bibitem{BoardmanJ:1995}
J.~M. Boardman, D.C. Johnson, and W.~S. Wilson.
\newblock Unstable operations in generalized cohomology.
\newblock In {\em Handbook of Algebraic Topology}, pages 687--828.
  North-Holland, Amsterdam, 1995.

\bibitem{BorgerW:2005}
J.~Borger and B.~Wieland.
\newblock Plethystic algebra.
\newblock {\em Adv. Math.}, 194(2):246--283, 2005.

\bibitem{Bousfield:1996}
A.~K. Bousfield.
\newblock On {$p$}-adic {$\lambda$}-rings and the {$K$}-theory of {$H$}-spaces.
\newblock {\em Math. Z.}, 223(3):483--519, 1996.


\bibitem{Hunton:2002}
J.~R. Hunton.
\newblock
     Complete cohomology theories and the homology of their omega
              spectra.
\newblock {\em Topology}, 41(5):931--943, 2002.
 

\bibitem{Mycroft:2017}
W.~Mycroft.
\newblock {\em Unstable Cohomology Operations: Computational Aspects of
  Plethories}.
\newblock PhD thesis, University of Sheffield, 2017.

\bibitem{Rezk:2009}
C.~Rezk.
\newblock The congruence criterion for power operations in {M}orava
  {$E$}-theory.
\newblock {\em Homology Homotopy Appl.}, 11(2):327--379, 2009.

\bibitem{StaceyW:2009}
A.~Stacey and S.~Whitehouse.
\newblock The hunting of the {H}opf ring.
\newblock {\em Homology Homotopy Appl.}, 11(2):75--132, 2009.

\bibitem{Strickland:2000}
N.~P. {Strickland}.
\newblock Formal schemes and formal groups.
\newblock arXiv:math/0011121, November 2000.

\bibitem{Sweedler:1969}
M.~Sweedler.
\newblock {\em Hopf algebras}.
\newblock W.A. Benjamin, New York, 1969.

\bibitem{TallW:1970}
D.~O. Tall and G.~C. Wraith.
\newblock Representable functors and operations on rings.
\newblock {\em Proc. London Math. Soc. (3)}, 20:619--643, 1970.

\bibitem{Yau:2003}
D.~Yau.
\newblock Unstable {$K$}-cohomology algebra is filtered {$\lambda$}-ring.
\newblock {\em Int. J. Math. Math. Sci.}, (10):593--605, 2003.

\bibitem{Yau:2010}
D.~Yau.
\newblock {\em Lambda-Rings}.
\newblock World Scientific, Singapore, 2010.

\end{thebibliography}
\end{document}